\documentclass[12pt]{amsart}
\usepackage[dvips]{graphicx}
\usepackage{amsmath,graphics}
\usepackage{amsfonts,amssymb}
\usepackage{xypic}
\theoremstyle{plain}
\newtheorem*{theorem*}{Theorem}
\newtheorem*{lemma*} {Lemma}
\newtheorem*{corollary*} {Corollary}
\newtheorem*{proposition*}{Proposition}
\newtheorem*{conjecture*}{Conjecture}
\newtheorem{theorem}{Theorem}[section]
\newtheorem{lemma}[theorem]{Lemma}
\newtheorem*{theorem1*}{Theorem 1}
\newtheorem*{theorem2*}{Theorem 2}
\newtheorem*{theorem3*}{Theorem 3}
\newtheorem{corollary}[theorem]{Corollary}
\newtheorem{proposition}[theorem]{Proposition}
\newtheorem{conjecture}[theorem]{Conjecture}
\newtheorem{question}[theorem]{Question}

\theoremstyle{remark}
\newtheorem*{remark}{Remark}
\newtheorem*{definition}{Definition}

\newtheorem{example*}{Example}

\theoremstyle{definition}

\def\gl{\mbox{GL}} \def\Q{\Bbb{Q}} \def\F{\Bbb{F}} \def\Z{\Bbb{Z}} \def\R{\Bbb{R}} \def\C{\Bbb{C}}
\def\N{\Bbb{N}}  \def\l{\lambda}  
 \def\a{\alpha} \def\g{\gamma}  \def\bp{\begin{pmatrix}}
\def\sm{\setminus} \def\ep{\end{pmatrix}} \def\bn{\begin{enumerate}} 
  \def\div{\mbox{div}} \def\en{\end{enumerate}}
\def\ba{\begin{array}} \def\ea{\end{array}}  
 \def\S{\Sigma} \def\s{\sigma} \def\a{\alpha} \def\b{\beta} \def\ti{\tilde}
  \def\im{\mbox{Im}} \def\sign{\mbox{sign}}
  
\def\be{\begin{equation}} \def\ee{\end{equation}} 
   
 \def\hom{\mbox{Hom}}  
   \def\k{\kilon}
 \def\dim{\mbox{dim}}

\def\zt{\Z[t^{\pm 1}]}    
\def\w{\omega}   
    \def\fr12{\frac{1}{2}} \def\z12{\Z[\fr12]}

\def\tpm {[t^{\pm 1}]}
\def\xpm {[x^{\pm 1}]}

\def\k{\kappa}

\begin{document}
\title[Twisted Alexander polynomials and $4$--manifolds]{Twisted Alexander polynomials, symplectic $4$--manifolds and  surfaces of minimal complexity}
\author{Stefan Friedl}
%\address{University of Warwick, Coventry, United Kingdom}
%\email{S.K.Friedl@warwick.ac.uk}
\address{\begin{tabular}{l} Universit\'e du Qu\'ebec \`a Montr\'eal, Montr\'eal, Qu\'ebec, and\\ University of Warwick, Coventry, UK\end{tabular}}
\email{sfriedl@gmail.com}
\author{Stefano Vidussi}
\address{Department of Mathematics, University of California,
Riverside, CA 92521, USA} \email{svidussi@math.ucr.edu}
\thanks{S. Friedl was  supported by a CRM--ISM Fellowship and by CIRGET}
\thanks{S. Vidussi was partially supported by a University of California Regents' Faculty Fellowships.}

\date{\today}
\begin{abstract}
Let $M$ be a $4$--manifold which admits a free circle action. We use twisted Alexander polynomials to study the existence of symplectic structures and the minimal complexity of surfaces in $M$. The results on the existence of symplectic structures summarize previous results of the authors in \cite{FV08a,FV08,FV07}.
The results on surfaces of minimal complexity are new.
\end{abstract}
\maketitle

\section{Introduction and main results}

%==========================================================
\subsection{$4$-manifolds with free circle action}
Let $M$ be a $4$--manifold which admits a  free circle action. (Throughout the paper, unless otherwise stated, we will assume that all
manifolds are closed, oriented and connected.)
In this paper we will discuss how twisted Alexander polynomials give information, for this class of manifolds, on two central problems in $4$--dimensional topology, namely the study of the existence of symplectic structures, and the minimal complexity of surfaces in $M$.
%Later in this introduction we will better specify the problems and summarize our results on the subject.

We start by recalling some simple facts about this class of  manifolds, as they will be frequently used in what follows. The existence of a free $S^1$--action on $M$ implies that $M$ is a
principal $S^1$--bundle, so that there is a projection map $p : M \to N$ where we denote by $N$
the orbit space of the free circle action. This principal bundle is determined by its Euler class
$e \in H^2(N;\Z)$. Note that $e=0$ if and only if $M$ is a product bundle $M=S^1\times N$. We have the Gysin sequence
\begin{equation}
\label{gysin} \xymatrix{  H^{0}(N;\Z) \ar[d]^\cong \ar[r]^{\cup \hspace*{1pt} e}& H^2(N;\Z) \ar[d]^\cong \ar[r]^{p^{*}}&
H^2(M;\Z) \ar[d]^\cong \ar[r]^{p_{*}} &H^{1}(N;\Z) \ar[d]^\cong \ar[r]^{\cup \hspace*{1pt} e} &H^3(N;\Z)\ar[d]^\cong \\
H_3(N;\Z)\ar[r]^{\cap e}&H_1(N;\Z)\ar[r]&H_2(M;\Z)\ar[r]^{p_*}&H_2(N;\Z)\ar[r]^{\cap e}&H_0(N;\Z),}
\end{equation}
where $p_* : H^{2}(M;\Z) \to H^{1}(N;\Z)$ denotes integration along the fiber.
From this sequence we can immediately deduce that \be \label{betti} b_2(M)=\left\{ \ba{rl} 2b_1(N),&\mbox{ if $e$ is torsion,}\\ 2b_1(N)-2,&\mbox{ if $e$ is non--torsion.}\ea \right.\ee
One can easily see that the intersection form vanishes on the half--dimensional space  $\im\{H_1(N;\Q)\to H_2(M;\Q)\}\subset H_2(M;\Q)$. It follows that $b_2^+(M) = b_2^-(M)$, so that $\sign(M)=0$.
For sake of simplicity,  we will restrict ourselves in this paper to the case that $e$ is either trivial or non--torsion. Also, we will assume that $b_2^+(M) > 1$, but the techniques and results presented here extend to the torsion case and, with the usual \textit{caveats}, to the case of $b_2^+(M) = 1$. We refer to \cite{FV08a,FV08,FV07} for the details of these cases.
Finally, note that the
long exact homotopy sequence of the fibration shows that the map $p_*:\pi_1(M)\to \pi_1(N)$ is an epimorphism.
%========================================
\subsection{Symplectic $4$--manifolds and fibered $3$--manifolds}
Remember that a 4--manifold is called \textit{symplectic} if it admits a closed,
non--degenerate 2--form $\w$. Our first goal is to study the question of  which 4--manifolds with free circle actions are symplectic. Let $M$  be a $4$--manifold as above. It is well--known that if  $(N,\phi)$ fibers over $S^1$ for some $\phi \in H^1(N,\Z)$ with $\phi\cup e=0$, then $M$ admits a symplectic structure. We refer to \cite{Th76}, \cite{Bou88}, \cite{FGM91} and \cite{FV07} for details. (Here we say that $(N,\phi)$ fibers over $S^{1}$ if the homotopy class of maps
$N \to S^1$ determined by $\phi \in H^{1}(N;\Z) = [N,S^{1}]$ contains a representative that is a fiber bundle
over $S^{1}$.)
It is natural to ask whether the converse of this statement holds true. We point out that the openness of the symplectic condition implies that if $M$ is symplectic then $M$ also has a symplectic form representing an integral cohomology class. We will implicitly make such a choice whenever necessary. Finally, if $\w$ is a symplectic form it follows immediately from $[\w]^2\ne 0 \in H^4(M;\R)$ that $p_*[\w]\ne 0\in H^1(N;\R)$.

The problem of studying the existence of symplectic forms on $M$
can be summarized in terms of the following conjecture:

\begin{conjecture} \label{conjfolk}
Let $M$ be a $4$--manifold with a free circle action with orbit space $N$. If $M$  admits a symplectic structure $\w$ with $[\w]\in H^2(N;\Z)$, then $(N,p_*[\w])$ fibers over $S^1$.
\end{conjecture}

This problem was first raised in  \cite{FGM91},
but only since the appearance of  Taubes' results on Seiberg-Witten invariants of symplectic $4$--manifolds (see \cite{Ta94,Ta95}) are  there  tools at  hand to seriously tackle the conjecture.
We refer to \cite{Kr98}, \cite{Kr99}, \cite{McC01}, \cite{Ba01}, \cite{Vi03} for results supporting this conjecture.

%==========================================================
\subsection{Surfaces of minimal complexity}
Given a surface $\S$ with connected components $\S_1,\dots,\S_k$ we define its complexity to be
\[ \chi^-(\S)=\sum_{i=1}^k \max\{-\chi(\S_i),0\}.\]
We say that an embedded surface $\S$ in a ($3$-- or $4$--) manifold has minimal complexity if it minimizes the complexity in its homology class.

Given a 3--manifold $N$ and $\phi \in H^1(N;\Z)$ the Thurston norm on $H^1(N,\Z)$ is defined as
\[ \chi^-_N(\phi)= \chi^-(\S)\]
where $\S$ is a surface of minimal complexity dual to $\phi$.
Given a 4--manifold $M$ and $\s \in H^2(M;\Z)$ we also define
\[ \chi^-_M(\s)= \chi^-(\S)\]
where $\S$ is a surface of minimal complexity dual to $\phi$. (Note that in spite of the similar definition, the minimal complexity function on $4$--manifolds is not known to share the properties of the Thurston norm, e.g. linearity. We refer to \cite[Section~7]{Kr98} for more related open questions.)

Let $M$ be a 4--manifold with $b_2^+(M) > 1$ and consider a class $\s \in H^2(M;\Z)$; then the adjunction inequality says that \be \label{equ:adj} \chi^-_M(\s) \geq \s\cdot \s +\s \cdot \k \ee
for any Seiberg-Witten basic class $\k \in H^2(M;\Z)$. Even though this inequality gives useful lower bounds, it is in general not enough to determine the function $\chi^-_M:H^2(M;\Z)\to \N$.
In this paper  we are interested in the study of this function for 4--manifolds $M$ with a free circle action.
Note that it is not difficult to see, in this case (cf. Section \ref{section:sw}), that we can put the absolute value to $\sigma \cdot \sigma$ in equation (\ref{equ:adj}).

Given $\s \in H^2(M;\Z)$ we say that $\s$ has property $(*)$ if there exists a
 (possibly disconnected) embedded  surface $\S \subset N$ and a (possibly disconnected) closed curve $c\subset N$ in general position with the following properties:
\bn
\item $\S$ is a Thurston norm minimizing surface dual to $p_*(\s)$,
\item given a lift $\ti{\S}$ of $\S$ to $M$ (which by Lemma \ref{lem:lift} always exists) the singular surface $p^{-1}(c)\cup \ti{\S}$ represents $PD(\s)$,
\item the geometric intersection number of $\S$ and $c$ is given by the absolute value of the algebraic intersection number $\S\cdot c$.
\en

Property $(*)$ is the suitable generalization of the property defined by Kronheimer in \cite{Kr99} for the product case.
We will see in Section \ref{section:surface} that  property $(*)$ is fairly common; in particular, it is satisfied whenever the Alexander polynomial $\Delta_{N,p_*(\s)}\ne 0\in \zt$. On the other hand we will also see that, for suitable $N$, there exist $\s$ which do not satisfy $(*)$.
In  Section \ref{section:surface} we prove the following lemma (cf. also \cite[Section~1.2]{Kr99}).

\begin{lemma} \label{lem:min}
Let $\s\in H^2(M;\Z)$ that satisfies $(*)$, then
\[ \chi^-_M(\s) \leq \chi^-_N(p_*(\s))+|\s\cdot \s|.\]
\end{lemma}

Kronheimer \cite[Corollary~2]{Kr99} \cite[Corollary~7.6]{Kr98} proved that in the case that $M$ is a product
$S^1\times N$, in many cases this is in fact an equality.
More precisely, Kronheimer proved the following result.

\begin{theorem} \label{thm:kronheimer}
Let $N$ be an irreducible 3--manifold whose Thurston norm is not totally degenerate, $M=S^1\times N$ and $\s\in H^2(S^1\times N;\Z)$. Then
\[   \chi^-_M(\s) \geq \chi^-_N(p_*(\s))+|\s\cdot \s|.\]
\end{theorem}

%As Kronheimer points out (see remark after \cite[Corollary~3]{Kr99}), the extra assumption $\chi^-_N(p_*(\s))>0$ is in many cases unnecessary.\footnote{In principle %one can now add a few examples of our own: if we have a Whitehead double
%of a knot where the twisted Alexander polynomial detects the genus, then with
%the methods below we will have found $\chi^-{S^1\times N(K)}$, but then perhaps these guys have smooth foliations anyway?}
Note that Kronheimer showed that  it is not always necessary to assume  that $N$ is irreducible (cf. \cite[Theorem~10]{Kr99}) and to assume that the Thurston norm is non--degenerate (cf.
\cite[Section~1.2]{Kr99}).
It now seems reasonable to pose the following question (cf. also \cite[Question~7.12]{Kr98}).

\begin{question} \label{conj:kronheimer}
Let $M\to N$ be a principal $S^1$--bundle.
Given $\s\in H^2(M;\Z)$, does the inequality
\[   \chi^-_M(\s) \geq \chi^-_N(p_*(\s))+|\s\cdot \s|\] always hold?
Is this an equality even for $\s$ which do not satisfy $(*)$?
\end{question}

Note that given
$\s\in H_2(M;\Z)$ we always have the inequality
\[   \chi^-_M(\s) \geq \chi^-_N(p_*(\s)).\]
This  follows from Gabai's result \cite{Ga83} that for any singular surface $S'\subset N$ dual to some $\phi\in H^1(N;\Z)$ we have
$\chi^-(S')\geq \chi^-_N(\phi)$.

%==========================================================
\subsection{Statement of the main results}

Given a 3--manifold $N$, $\phi \in H^1(N;\Z)$ and an epimorphism $\a:\pi_1(N)\to G$ onto a finite group we can define the $1$--variable twisted Alexander polynomial $\Delta_{N,\phi}^\a \in \zt$. We refer to Section \ref{section:twialex} for details. Given a non--zero Laurent polynomial $\sum_{i=n}^{m}a_it^i$ with $a_n \neq 0$ and $a_m \neq 0$ we define \[ \mbox{deg} \big(\sum_{i=n}^{m}a_it^i\big) = m - n; \] furthermore we define $\mbox{deg}(0) = - \infty$. Given an integer $n$ we write $- \infty + n = - \infty$ and
$- \infty < n$.

We will show that the degrees of twisted Alexander polynomials give lower bounds on $\chi^-_M(\s)$
for a given $\s\in H^2(M;\Z)$.
Note that if $p_*(\s)=0\in H^1(N;\Z)$, then $\s$ is clearly dual to a union of  embedded tori and hence  $\chi^-_M(\s)=0$. We therefore restrict ourselves now to the case that $\phi=p_*(\s)\ne 0 \in H^1(N;\Z)$.

Given $\phi \in H^1(N,\Z) = \mbox{Hom}(\pi_1(N),\Z)$ and a homomorphism $\a:\pi_1(N)\to G$, we denote by $\phi_{\a}$ its restriction to $\mbox{ker} \alpha \subset \pi_1(N)$. Furthermore we denote by $\div \phi_\a$ the divisibility of $\phi_\a$, i.e. the largest integer $n$ such that $\frac{1}{n}\phi_\a$ still defines an integral class. Our main theorem is now the following.

\begin{theorem} \label{thm:joint} Let $M$ be a 4--manifold admitting a free circle
action  such that $b_2^+(M) > 1$ and such that either $e=0$ or $e$ is non--torsion. Let $\s \in H^{2}(M;\Z)$ such that $\phi=p_*(\s)\ne 0 \in H^1(N;\Z)$.
 Then for any epimorphism $\alpha : \pi_1(N) \to G$ onto a finite group we have
 \[ \chi^-_M(\s) \geq \frac{1}{|G|}( \mbox{deg}(\Delta_{N,\phi}^{\a}) - 2 \mbox{div}\, \phi_{\a}) +|\s\cdot \s|.\]
Furthermore, if $\s$ is represented by a symplectic form, then $\Delta_{N,\phi}^{\alpha}$ is monic.
\end{theorem}

Note that the last statement is already contained in \cite{FV08a} and \cite{FV07}.
The proof of the theorem relies on the adjunction inequality and Taubes' results applied to finite coverings of $M$. The resulting information on Seiberg--Witten invariants can be translated into information on twisted Alexander polynomials using the results of Meng and Taubes \cite{MT96}, Baldridge \cite{Ba01}, \cite{Ba03} and Shapiro's lemma.

This theorem allows us to study the existence of symplectic structures and the complexity of surfaces using twisted Alexander polynomials.
In \cite{FV08a}, \cite{FV08} and \cite{FV07} the authors showed that twisted Alexander polynomials are very efficient at detecting fibered 3--manifolds
(cf. also \cite{Ch03}, \cite{GKM05}, \cite{FK06} and \cite{Ki08}). The main result therein contained is the following.
%In particular we proved  the following result. (We refer to \cite{FV08,FV07} for details on the case $b_2^+(M)=1$.)

\begin{theorem} \label{thm:enough}
Let $N$ be a  3--manifold.
Let $\phi \in H^1(N;\Z)$ be non--trivial. If $\Delta_{N,\phi}^\a\ne 0$ for all epimorphisms $\a:\pi_1(N)\to G$ onto finite groups, then $N$ is prime.
Furthermore, if $N$ has either vanishing Thurston norm or is a graph manifold, then  $(N,\phi)$ fibers over $S^1$.
\end{theorem}

The primeness conclusion has first been proved by McCarthy \cite{McC01}.  In particular the combination of Theorems \ref{thm:joint} and \ref{thm:enough} gives an affirmative answer to Conjecture \ref{conjfolk} if the orbit space $N$ has vanishing Thurston norm or if $N$ is a graph manifold. (For the case of vanishing Thurston norm, a proof of the Conjecture along the same lines is presented in \cite{Bow07}.)

Whereas twisted Alexander polynomials are good at detecting fibered 3--manifolds,  their record at detecting the Thurston norm of a given $\phi\in H^1(N;\Z)$ is rather mixed. On the one hand we prove  in Section \ref{section:examples} the following result.

\begin{theorem}
There exists  a 4--manifold $M$ with free circle action with non--torsion Euler class and $\s\in H^2(M;\Z)$ where the
 adjunction inequality is not strong enough to determine $\chi^-_M(\s)$, but where the
 bounds from Theorem \ref{thm:joint} on $\chi^-_M(\s)$ coming from $\Delta_{N}^\a$ for an appropriate $\a:\pi_1(N)\to G$ determine $\chi^-_M(\s)$.
 \end{theorem}

 On the other hand we will see in Lemma \ref{lem:bad} that if the ordinary Alexander polynomial $\Delta_{N,\phi}$ vanishes, then all twisted Alexander polynomials $\Delta_{N,\phi}^\a$ vanish as well. Note that this effect would not happen if we used twisted Alexander polynomials corresponding
  to general representations $\a:\pi_1(N)\to \gl(\C,k)$ as in \cite{FK06}. Unfortunately these more general twisted Alexander polynomials seem to have no interpretation in terms of Seiberg--Witten invariants of covers.\\

\noindent \textbf{Acknowledgment.} The first author would like to thank the organizers of the Postnikov memorial conference in Bedlewo, Poland for the relaxed atmosphere which allowed for many interesting conversations. We also would like to thank the referee for various helpful comments.

%==========================================================
\section{Preliminaries}

%==========================================================

%==========================================================
\subsection{Surfaces and principal $S^1$--bundles}\label{section:surface}
We first prove the following lemma regarding surfaces in $N$ and $M$.

\begin{lemma} \label{lem:lift}
Let $M\to N$ be a principal $S^1$--bundle with Euler class $e$.
Let $\s \in H^2(M;\Z)$ and $\S\subset N$ an embedded surface dual to $p_*(\s)$.
Then $\S$ lifts to a surface in $M$.
\end{lemma}

\begin{proof}
Clearly it is enough to show that the principal $S^1$--bundle over $N$ restricted to $\S$ is trivial. This in turn is equivalent to showing that
$e|_\S=0 \in H^2(\S;\Z)$. But since $\S$ is dual to $p_*(\s)$, this is equivalent to the condition that $p_*(\s) \cup e=0$,
which in turn is satisfied by the Gysin sequence (\ref{gysin}).
 \end{proof}

%
%
%Recall that we say that $\s\in H^2(M;\Z)$ has property $(*)$ if there exists an embedded  surface $\S \subset N$ and a (possibly disconnected)  closed curve $c\subset N$ in general position with the following properties:
%\bn
%\item $\S$ is a Thurston norm minimizing surface dual to $p_*(\s)$,
%\item given a lift $\ti{\S}$ of $\S$ to $M$ (which by Lemma \ref{lem:lift} always exists) the singular surface $p^{-1}(c)\cup \ti{\S}$ represents $PD(\s)$,
%\item the geometric intersection number of $\S$ and $c$ is given by the absolute value of the algebraic intersection number $\S\cdot c$.
%\en

We have the following  lemma (cf. also \cite{Kr99}).

\begin{lemma}  \label{lem:starmin}
Let $\s\in H^2(M;\Z)$ that satisfies $(*)$, then
\[ \chi^-_M(\s) \leq \chi^-_N(p_*(\s))+|\s\cdot \s|.\]
\end{lemma}

\begin{proof}
Let $\S,\ti{\S}$ and $c$ as in the definition of property ($*$).
 At each singular point  of $p^{-1}(c)\cup \ti{\S}$ we can replace a pair of transverse disks with an embedded annulus having the same oriented boundary. Note that each replacement increases the Euler number by 2. We therefore obtain a smooth surface $T$ representing the class dual to $\s$ with
\[ \chi^-(T)=\chi^-(\ti{\S})+\chi^-(p^{-1}(c))+2 |\S \cdot c|.\]
Note that $p^{-1}(c)$ is a union of tori, hence $\chi^-(p^{-1}(c))=0$. As
$2~\S\cdot c=\s\cdot \s$, we get
\[ \chi^-(T)=\chi^-_N(p_*(\s))+|\s \cdot \s|.\]
\end{proof}

We have the following criterion for $\s\in H^2(M;\Z)$ having property $(*)$.

\begin{lemma} \label{lem:propstar}
Let $\s\in H^2(M;\Z)$ and let $n=\div(p_*(\s))$ be the divisibility of $p_*(\s)\in H^1(N;\Z)$. If there exists a connected Thurston norm minimizing surface dual to the primitive class $\frac{1}{n}p_*(\s)$, then $\s$ has property $(*)$.
\end{lemma}

\begin{proof} Let $\s\in H^2(M;\Z)$ and write $n=\div(p_*(\s))$. Assume $\frac{1}{n}p_*(\s)$ is dual to a connected  Thurston norm minimizing  surface $\S'$.
By \cite{Th86} the union of $n$ parallel copies of $\S'$ is then a Thurston norm minimizing  surface dual to $p_*(\s)$.
By Lemma \ref{lem:lift} we can lift $\S'$ to a surface $\ti{\S}'\subset M$.
It follows from the Gysin sequence (\ref{gysin}) that we can find an embedded curve $c\in N$ such that the class dual to
$\s$ is given by $[n\ti{\S}']+[p^{-1}(c)]$. We can assume that $c$ is in general position with $n\S'$, hence $n\ti{\S}'$ and $p^{-1}(c)$ are in general position.
Since $\S'$ is connected we can  choose $c$ such that the geometric intersection number of $\S'$ and $c$ is given by the absolute value of $\S'\cdot c$.
(Note that $c$ could be disconnected.)
\end{proof}

The following  lemma  follows easily from \cite[Proposition~6.1]{McM02} or \cite{Tu02}.

\begin{lemma} Let $\phi \in H^1(N;\Z)$ and $n=\div(\phi)$. If $\Delta_{N,\phi}\ne 0$, then there exists a connected Thurston norm minimizing surface dual to the primitive class $\frac{1}{n}\phi$.
\end{lemma}

\begin{remark}
We now give examples of manifolds of the form $M=S^1\times N$ such that
not every $\s \in H^2(M;\Z)$ has property $(*)$. For example, let $N$ be the connected sum of the zero framed surgeries on two non--trivial oriented knots $K_1$ and $K_2$. Let $\mu_1,\mu_2\in H^1(N;\Z)$ be a basis given by the meridians of $K_1$ and $K_2$. Furthermore let $F_1,F_2$ be the result of capping off  two minimal genus Seifert surfaces of $K_1$ and $K_2$.
Consider $\phi=PD(a_1[F_1])+PD(a_2[F_2]), a_1,a_2\in \Z$, then for $a_1\ne 0$ and $a_2\ne 0$ there exists no connected Thurston norm minimizing surface dual to $\phi$.

We now specialize to $a_1=3$ and $a_2=2$.
Also, let  $\g=\mu_1-\mu_2$. Then $\g\cdot \phi=1$. But it is easy to see that there exists no Thurston norm minimizing surface $\S$ dual to $\phi$ and a curve $c$ representing $\g$ such that the geometric intersection number of $\S$ and $c$ equals $1$. Now let $\s\in H^2(S^1\times N;\Z)$ be the element which corresponds to
$PD(\g)+\phi \otimes 1$ under the K\"unneth decomposition $H^2(S^1\times N;\Z)=H^2(N;\Z)\oplus H^1(N;\Z)\otimes H^1(S^1;\Z)$.
Then it is clear that $\s$ does not have property ($*$).

Finally, let  $\g=\mu_1+\mu_2$. Then $\g\cdot \phi=5$ and we can find a  Thurston norm minimizing surface $\S$ dual to $\phi$ and a curve $c$ representing $\g$ such that the geometric intersection number of $\S$ and $c$ equals $5$. The corresponding element $\s\in H^2(S^1\times N;\Z)$ has property ($*$), but does not satisfy  the conditions of Lemma \ref{lem:propstar}.

These examples show that in general it is not clear when ($*$) is satisfied. Also, there does not seem to be a good conjecture for what  $\chi_M^-(\s)$ should be in the case that ($*$) does not hold.
\end{remark}

%==========================================================================
\subsection{Twisted Alexander polynomials} \label{section:twialex}

In this section we are going to recall the definition of the (twisted) Alexander polynomial associated to an epimorphism of the fundamental group of a compact $3$--manifold onto a finite group. Twisted Alexander polynomials were first
introduced for the case of knots by Xiao--Song Lin \cite{Li01} and Lin's definition was later generalized to 3--manifolds by Wada \cite{Wa94}, Kirk--Livingston \cite{KL99} and Cha \cite{Ch03}.

Let $N$ be a compact $3$--manifold,  $\phi \in H^{1}(N;\Z)=\hom(\pi_1(N),\Z)$ and let $\a:\pi_{1}(N) \to G$ be an
epimorphism onto a finite group $G$. Then $\a\times \phi$  gives an action of $\pi_{1}(N)$ on
$G \times \Z$, which extends to a ring homomorphism from
$\Z[\pi_{1}(N)]$ to the  $\Z[t^{\pm 1}]$--linear endomorphisms of
$\Z[G \times \Z] = \Z[G]\tpm$. This  induces a  left
$\Z[\pi_{1}(N)]$--structure
 on $\Z[G]\tpm$.

Now let $\ti{N}$ be the universal cover of $N$. Note that
$\pi_{1}(N)$ acts on the left on $\ti{N}$ as group of deck
transformation. The chain groups $C_*(\ti{N})$ are in a natural way
right $\Z[\pi_1(N)]$--modules, with the right action on
$C_{*}(\ti{N})$ defined via $\sigma \cdot g := g^{-1}\sigma$, for
$\sigma \in C_{*}(\ti{N})$. We can form by tensor product the chain
complex $C_*(\ti{N})\otimes_{\Z[\pi_1(N)]}\Z[G]\tpm$. Now define
$H_{i}(N;\Z[G]\tpm):=
H_i(C_*(\ti{N})\otimes_{\Z[\pi_1(N)]}\Z[G]\tpm)$, which inherit the
structure of $\Z\tpm$--modules. These module are called
twisted Alexander modules.

Our goal is to define an invariant out of $H_{1}(N;\Z[G]\tpm)$.
First note that by endowing $N$ with a finite cell structure we can
view
 $C_*(\ti{N})\otimes_{\Z[\pi_1(N)]}\Z[G]\tpm$
as finitely generated $\Z\tpm$--modules. The $\Z\tpm$--module
$H_1(N;\Z[G]\tpm)$ is now a finitely presented and finitely related
$\Z\tpm$--module since $\Z\tpm$ is Noetherian. Therefore
$H_1(N;\Z[G]\tpm)$ has a free $\Z\tpm$--resolution
\[ \Z\tpm^r \xrightarrow{S} \Z\tpm^s \to H_1(N;\Z[G]\tpm) \to 0 \]
of finite $\Z\tpm$--modules. Without loss of generality we can
assume that $r\geq s$.

\begin{definition} \label{def:alex} The \emph{twisted Alexander polynomial} of $(N,\a,\phi)$ is defined
to be the order of the $\Z\tpm$--module $H_1(N;\Z[G]\tpm)$, i.e. the
greatest common divisor of the $s\times s$ minors of the $s\times
r$--matrix $S$. It is denoted by $\Delta_{N,\phi}^{\a}\in \Z\tpm$,
and it is well--defined up to units of $\Z\tpm$.
\end{definition}

%It is well--known (see e.g. \cite{FV08a}) that, up to sign, there is
%a unique choice of $\Delta_{N,\phi}^{\a} \in \Z\tpm$ symmetric under
%the natural involution of $\Z\tpm$.

If $G$ is the trivial group
we will drop $\a$ from the notation. With these conventions,
$\Delta_{N,\phi} \in \Z[t^{\pm1}]$ is the ordinary $1$--variable
Alexander polynomial associated to $\phi$. For example,
if $X(K)=S^3\sm \nu K$ is the exterior of a knot $K$  and $\phi\in H^1(X(K);\Z)$ is a generator, then  $\Delta_{X(K),\phi}$ equals the ordinary Alexander polynomial $\Delta_K$ of a knot.

Finally, given a 3--manifold $N$ we write $H=H_1(N;\Z)/\mbox{torsion}$. Using a similar approach as above one can define the multivariable Alexander polynomial $\Delta_N\in \Z[H]$. The following theorem of Meng and Taubes \cite{MT96} states that the multivariable Alexander polynomial of a 3--manifold $N$ corresponds to the Seiberg--Witten invariants of $N$.

\begin{theorem} \label{thm:mt}
Let $N$ be a closed 3--manifold with $b_1(N) > 1$ and let $H=H_1(N;\Z)/\mbox{torsion}$.
Then
\[ \Delta_N = \pm \sum_{\xi \in H^{2}(N)} SW_{N}(\xi) \cdot \frac12 f(\xi) \in \Z[H],
\]where $f$ denotes the composition of Poincar\'e duality with the quotient map $f: H^2(N) \cong
H_{1}(N) \to H$ and, as $f(\xi)$ has even divisibility for all $3$--dimensional basic classes $\xi
\in \mbox{supp\,} SW_{N}$, multiplication by $\frac12$ is well-defined.
\end{theorem}

%\begin{remark} The $1$-variable twisted Alexander polynomial defined
%above can also be described as the  specialization of a
%multivariable twisted Alexander polynomial taking values in $\Z[H]$,
%where $H$ is the maximal free quotient  of $H_1(N)$. This polynomial, in turn, is
%related with the ordinary Alexander polynomial of the $G$--cover
%$N_{G}$ of $N$ and then, thanks to \cite{MT96}, to the
%Seiberg-Witten invariants of $S^1 \times N_G$. These observations
%constitute the starting point of the connection between Conjecture
%\ref{conjcha} and Conjecture \ref{conjfolk}. See \cite{FV08a} for
%details. \end{remark}

%==========================================================
\section{Constraints from Seiberg-Witten theory} \label{sw}

%===============================
\subsection{Seiberg-Witten theory for  manifolds with circle action} \label{ssw} \label{section:sw}
The essential ingredient  in our approach is the fact that the Seiberg-Witten invariants of $M$ are
related to the Alexander polynomial of $N$. The following theorem combines the results of Meng--Taubes \cite{MT96} and Baldridge  \cite[Corollaries~25~and~27]{Ba03} (cf. also \cite{Ba01}), to which we refer the reader for definitions and results for Seiberg-Witten theory in this set-up:

\begin{theorem} \label{bald} \label{thm:baldridge}  Let $M$ be a $4$--manifold with $b_2^+(M) > 1$ admitting a free circle action with  orbit
space $N$. Let $e\in H^2(N;\Z)$ be the  Euler class. Assume that either $e=0$ or $e$ non--torsion. Then the Seiberg-Witten invariant $SW_{M}(\kappa)$ of a class $\kappa = p^{*} \xi \in p^*
H^{2}(N;\Z) \subset H^{2}(M;\Z)$ is given by the formula \begin{equation} \label{baldfor} SW_{M}(\kappa)
=\sum_{\xi \in (p^*)^{-1}(\kappa)} SW_N(\xi)= \sum_{\xi' - \xi \, \equiv \, 0 \, (e)} SW_{N}(\xi') \in \Z. \end{equation}
Furthermore, $SW_M(\kappa)=0$ for any $\kappa \not\in p^* H^2(N;\Z)$.
\end{theorem}

%In the remainder of the paper, we will implicitly assume the geometrization conjecture. Without that
%assumption, the condition that $N$ is irreducible must be added to the hypotheses of (some of) our results.

In the formula above, $SW_{N}(\xi)$ is the $3$--dimensional SW--invariant of a class $\xi \in
H^{2}(N)$, and the effect of the twisting of the $S^1$--fibration, measured by the class $e \in
H^2(N)$, is to wrap up the contribution of all $3$--dimensional basic classes of $N$ that have the
same image in $H^2(M)$, i.e. that differ by a multiple of $e$. As usual, we can package the above
invariants in terms of a Seiberg-Witten polynomial.

From the calculation of the Seiberg--Witten invariants of $M$ we obtain immediately the following corollary.

\begin{corollary}
Let $M$  be a $4$--manifold with $b_2^+(M) > 1$ admitting a free circle action with  orbit
space $N$ such that the  Euler class is  either zero or non--torsion. Then
 \be \label{equ:adj2} \chi^-_M(\s) \geq |\s\cdot \s| +\s \cdot \k \ee
for any Seiberg-Witten basic class $\k \in H^2(M;\Z)$.
\end{corollary}

\begin{proof}
By (\ref{equ:adj}) we only have to consider the case that $\s\cdot \s<0$. Let $\k$ be a basic class of $M$. Recall that this implies that $-\k$ is also a basic class. Let $\varphi:M\to \hat{M}$ be the orientation reversing diffeomorphism given by $\varphi(p)=p$.
Given $a,b\in H^2(M;\Z)$ we have $\varphi(a)=a,\varphi(b)=b$ and $Q_{\hat{M}}(\varphi(a),\varphi(b))=-Q_M(a,b)$, where for sake of understanding we write explicitly the intersection forms of each manifold.

It follows from Theorem \ref{thm:baldridge} that $a$ is a basic class for $M$ if and only if $\varphi(a)$ is a basic class for $\hat{M}$.
Applying (\ref{equ:adj}) to $\varphi(\s)$ we obtain
\[ \ba{rcl} \chi^-_M(\s) &=& \chi^-_{\hat{M}}(\varphi(\s))\\
&\geq& Q_{\hat{M}}(\varphi(\s),\varphi(\s)) +Q_{\hat{M}}(\varphi(\s),\varphi(-\k))\\
&=& -Q_M(\s,\s) -Q_M(\s,-\k) \\
&=&|Q_M(\s,\s)|+Q_M(\s,\k).\ea \]
\end{proof}

%In \cite{FV07} we show that the canonical class of the symplectic structure satisfies the
%following (note that the case $b_2^+(M)>1$ also follows from the previous theorem.

%\begin{proposition} \label{prop:pullback} Let $(M,\omega)$ be a symplectic manifold admitting a free circle
%action with nontorsion Euler class $e \in H^2(N)$, where $N$ is the orbit space. Then the
%canonical class $K \in H^2(M)$ of the symplectic structure is the pull-back of a class $\zeta \in
%H^2(N)$, well--defined up to the addition of a multiple of $e$.\end{proposition}

%===============================
\subsection{Twisted Alexander polynomials and SW--invariants}

We are in position now to prove our main theorem. Note that the second part is already contained
in \cite{FV08a} and \cite{FV07}.

\begin{theorem}
  Let $M$ be a $4$--manifold with $b_2^+(M) > 1$ admitting a free circle action with  orbit
space $N$. Let $e\in H^2(N;\Z)$ be the  Euler class. Assume that either $e=0$ or $e$ non--torsion.
 Let $\s \in H^{2}(M;\Z)$ such that $\phi=p_*(\s)\ne 0 \in H^1(N;\Z)$.
 Then for any epimorphism $\alpha : \pi_1(N) \to G$ onto a finite group we have
 \[ \chi^-_M(\s) \geq \frac{1}{|G|}\left(\deg(\Delta_{N,\phi}^{\alpha})-2 \div \, \phi_{\a}\right)+|\s\cdot \s|.\]
Furthermore, if $\s$ is represented by a symplectic form, then
$\Delta_{N,\phi}^{\alpha}$ is monic.
\end{theorem}

\begin{proof}
%Our
%goal is to apply the adjunction inequality and Taubes' results (\cite{Ta94,Ta95}) on the Seiberg--Witten invariants of $M$ to impose constraints on
%the twisted Alexander polynomials of $N$.
 Let $M$ be a 4--manifold admitting a free circle
action  such that $b_2^+(M) > 1$ and denote by $N$ its orbit space.
It follows from equation (\ref{betti}) and the  remarks which follow (\ref{betti}) that $b_1(N)\geq 2$.
We will first analyze the  ordinary
$1$--variable Alexander polynomial $\Delta_{N,\phi}$. By \cite{FV08a} we can write this polynomial
as \be \label{equ:onemulti} \Delta_{N,\phi} = (t^{div {\phi}} - 1)^2 \cdot \sum_{g \in H} a_{g}
t^{\phi (g)} \in \Z[t^{\pm 1}], \ee where $H$ is the maximal free abelian quotient of $\pi_{1}(N)$
and $\Delta_{N} = \sum_{g \in H} a_{g} \cdot g \in \Z[H]$ is the ordinary multivariable Alexander
polynomial of $N$. By Theorem \ref{thm:mt} we can write
\be \label{inter} \Delta_{N,\phi} = \pm (t^{div {\phi}} - 1)^2 \sum_{\xi\in H^2(N)} SW_N(\xi)t^{\frac{1}{2}\phi \cdot \xi}. \ee
We will use now Equation (\ref{baldfor}) to write $\Delta_{N,\phi}$ in terms of the $4$--dimensional Seiberg-Witten invariants of $M$.
In order to do so, observe that for all classes $\xi \in H^2(N)$ we can write $\xi \cdot \phi = \xi \cdot p_{*} (\s) = p^{*}( \xi) \cdot \s = \kappa \cdot \s$ where $\kappa
 = p^{*}( \xi)$. Grouping together the contributions of the $3$--dimensional basic classes in terms of their image in $H^2(M)$, and using (\ref{baldfor}) we get
\be \label{equ:alexsw}  \ba{rcl} \Delta_{N,\phi} & = & \pm (t^{div \, \phi}-1)^2 \sum\limits_{\kappa \in p^{*} H^2(N)}\hspace{0.3cm} \sum\limits_{\xi\in (p^*)^{-1}(\kappa)} SW_N(\xi) t^{\frac{1}{2}\phi \cdot \xi} \\ \\ & = & \pm (t^{div \, \phi}-1)^2 \sum\limits_{\kappa \in p^{*} H^2(N)} SW_M(\kappa)t^{\frac{1}{2} \s \cdot
\kappa }.\ea \ee
Note that $\k$ is a basic class if and only if $-\k$ is a basic class.
It now follows that
\[  \ba{rl}& \mbox{max}\{ \k\cdot \s | \k \mbox{ basic  class of $M$}\}\\
 =&\mbox{max}\{ \k\cdot \s | \k \mbox{ basic  class of $M$ and }\k\in p^*(H^2(N;\Z))\}\\
\geq&\deg(\Delta_{N,\phi})-2\div(\phi).\ea \]
Combining this inequality with the adjunction inequality (\ref{equ:adj2}) we get
\be \label{equ:inequ2} \chi^-_M(\s) \geq \deg(\Delta_{N,\phi})-2 \div ( \phi)+|\s\cdot \s|.\ee

Now assume that $\s$ is represented by a symplectic form $\w$.
Taubes' constraints,
applied to the symplectic manifold $(M,\w)$, assert that if $K \in H^{2}(M)$ is the canonical class, then $SW_{M}(-K)
= 1$. Moreover, among all basic classes $\kappa \in H^2(M)$, we have \begin{equation}
\label{more} - K \cdot \s \leq \kappa \cdot \s, \end{equation} with
equality possible only for $\kappa = - K$.
%(When $b^{+}_{2}(M) = 1$, this statement applies to the Seiberg-Witten invariants evaluated in
%Taubes' chamber, but as remarked in Theorem \ref{bald} this specification is not a concern in our situation.)
 It now follows immediately from (\ref{equ:alexsw}) that $\Delta_{N,\phi}$ is a monic
polynomial.
%, and remembering the symmetry of $SW_N$ (or $\Delta_{N,\phi}$), we see
%that its Laurent degree is $d = K \cdot \s + 2 \div \phi = \zeta \cdot \phi + 2 \div \phi$.

Now let $\a:\pi_1(N)\to G$ be an
epimorphism onto  a finite group $G$. We denote by $\pi: N_{\a} \to N$ the corresponding regular $G$--cover of $N$.
It is well--known that $b_1(N_{\a}) \geq b_1(N)$.  The
epimorphism $\pi_1(M) \to \pi_{1}(N) \to G$ determines a regular $G$--cover of $M$ that we will
denote (with slight abuse of notation) $\pi: M_{\a} \to M$. These covers are related by the
commutative diagram
\begin{equation} \label{commie}
\begin{array}{ccc} M_\a & \stackrel{\pi}{\longrightarrow} & M \\ \downarrow &  & \downarrow \\ N_\a &
\stackrel{\pi}{\longrightarrow} & N \end{array} \end{equation} where the principal
$S^1$--fibration $p_{\a} : M_{\a} \to N_{\a}$ has Euler class $e_{\a} = \pi^{*} e \in
H^{2}(N_\a)$. Note that $e_\a=0$ whenever $e=0$, and $e_\a$ is non--torsion if and only if $e$ is non--torsion.
In particular $b_2^+(M_G)\geq b_2^+(M) > 1$.

Now let $\s \in H^2(M;\Z)$ and $\S\subset M $ a surface of minimal complexity representing $\s$. Note that $\pi^{-1}(\S)\subset M_\a$ is dual to $\pi^*(\s)$.
We  have
\[ \ba{rcl} \chi^-_M(\s)=\chi^-(\S)&=&\frac{1}{|G|}\chi^-(\pi^{-1}(\S))\geq \frac{1}{|G|} \chi^-_{M_\a}(\pi^*\s)\\[0.1cm]
\s \cdot \s=\S\cdot \S&=&\frac{1}{|G|}(\pi^{-1}(\S)\cdot \pi^{-1}(\S))= \frac{1}{|G|}(\pi^*(\s)\cdot \pi^*(\s)).\ea   \]
(Note that is not known whether or not $\pi^{-1}(\S)$ is a surface of minimal complexity.)
Applying (\ref{equ:inequ2}) to $M_\a$ it now follows that
\[ \ba{rcl} \chi^-_M(\s)&\geq& \frac{1}{|G|} \chi^-_{M_\a}(\pi^*(\s)) \\[0.1cm]
&\geq & \frac{1}{|G|} \left(\deg(\Delta_{N_\a,\phi_\a})-2 \div ( \phi_\a)+|\pi^*(\s)\cdot \pi^*(\s)|\right)\\[0.1cm]
&= & \frac{1}{|G|} \left(\deg(\Delta_{N_\a,\phi_\a})-2 \div ( \phi_\a)\right)+| \s \cdot \s|.\ea \]
This, together with the relation
$\Delta_{N,\phi}^{\a} = \Delta_{N_\a,\phi_\a}$ from \cite{FV08a} proves the first part of the theorem.

Now assume that $\s$ is represented by a symplectic structure $\w$.
As $(M,\omega)$ is symplectic, $M_{\a}$ inherits a symplectic form $\omega_{\a} :=
\pi^{*}\omega$ which represents $\pi^*(\s)$. Clearly $\phi_\a=p_*(\pi^*(\s))$. From the above we get that $\Delta_{N_\a,\phi_\a}$ is monic. Again the equality
$\Delta_{N,\phi}^{\a} = \Delta_{N_\a,\phi_\a}$ concludes the proof. \end{proof}

%==========================================================
\section{Examples} \label{section:examples}

%==========================================================
\subsection{Applications of Theorem \ref{thm:joint}}
%Furthermore this implies that $\pi_1(X(K))^{(1)}$
%is perfect, i.e., $\pi_1(X(K))^{(n)}=\pi_1(X(K))^{(1)}$ for any $n>1$. (For a group $G$, $G^{(n)}$
%is defined inductively as follows: $G^{(0)} := G$ and $G^{(n+1)} := [G^{(n)},G^{(n)}]$.) Therefore
%the genus bounds of Cochran \cite{Co04} and Harvey \cite{Ha05} vanish as well.

Applications of Seiberg--Witten invariants to the existence of symplectic
structures on 4--manifolds with a free circle action have been studied
by many authors. In the case of trivial Euler class we refer to
\cite{Kr98,Kr99,McC01,Vi03,FV08a}, in the case of non--trivial Euler class the
first results were obtained by Baldridge \cite{Ba01}.
Furthermore  in \cite{FV08a} and \cite{FV07}
we give many explicit examples of 4--manifolds with a free circle action where twisted Alexander polynomials can be used to show that they do  not support a symplectic structure.
In this section we therefore concentrate on examples regarding the minimal complexity of surfaces in 4--manifolds with a free circle action.

Let $T$ be the 3--torus. Recall that the 3--torus has vanishing  Thurston norm and that its multivariable Alexander polynomial is $1$.
 Let $x,y,z\in H_1(T;\Z)$ be a basis corresponding to the three circles of $T=S^1\times S^1\times S^1$
and let $C\subset T$ be a circle representing $x$. Throughout this section we denote by $\phi \in H^1(N;\Z)$ the class given by $\phi(x)=1, \phi(y)=\phi(z)=0$.
Pick a meridian $\mu_C$ and a longitude $\l_C$ for $C$ such that $[\mu_C]=0$ and $[\l_C]=x$ in $H_1(T;\Z)$. Next, let $K\subset S^3$ be an oriented knot. We denote by $\mu_K$ and $\l_K$ its
meridian and longitude. Now, splice the two exteriors to form the $3$--manifold
\begin{equation} \label{surge} T_K = (T\sm \nu C)\cup (S^3\sm \nu K) \end{equation}
where the gluing map on the boundary $2$--tori identifies $\mu_K$ with $\l_C$ and $\l_K$ with
$\mu_C$.
%Note that, if $N_{K}$ is the $0$--surgery of $S^3$ along $K$, and if  $m$ is the image of
%the canonically framed curve $\mu_K$ in $N_K$,  then we can write $X_K$ as normal connected sum
%of $X$ and $N_K$, i.e. $X_K = X \#_{C = m} N_{K}$.

As the surgery of Equation (\ref{surge}) amounts to the substitution of a solid torus with a
homology solid torus, respecting the boundary maps, and as the class of $C$ is primitive, it is
easy to see from the Mayer--Vietoris sequence that the inclusion maps induce isomorphisms
$H_1(T;\Z)\xleftarrow{\cong} H_1(T\sm \nu C;\Z)\xrightarrow{\cong} H_1(T_K;\Z)$ which we use to
identify these groups for the remainder of this section. We also identify $H^1(T;\Z)=H^1(T\sm \nu C;\Z)=H^1(T_K;\Z)$.

 Let $e=PD(z)$ and let $M_K(e)$ be the total space of the principal $S^1$--bundle over $T_K$ with Euler class $e$.
We now specialize to the case that  $K$ is the Conway knot $11_{401}$, its  diagram is given in Figure
\ref{example11n34}.
 \begin{figure}[h] \begin{center}
\begin{tabular}{cc}
\includegraphics[scale=0.27]{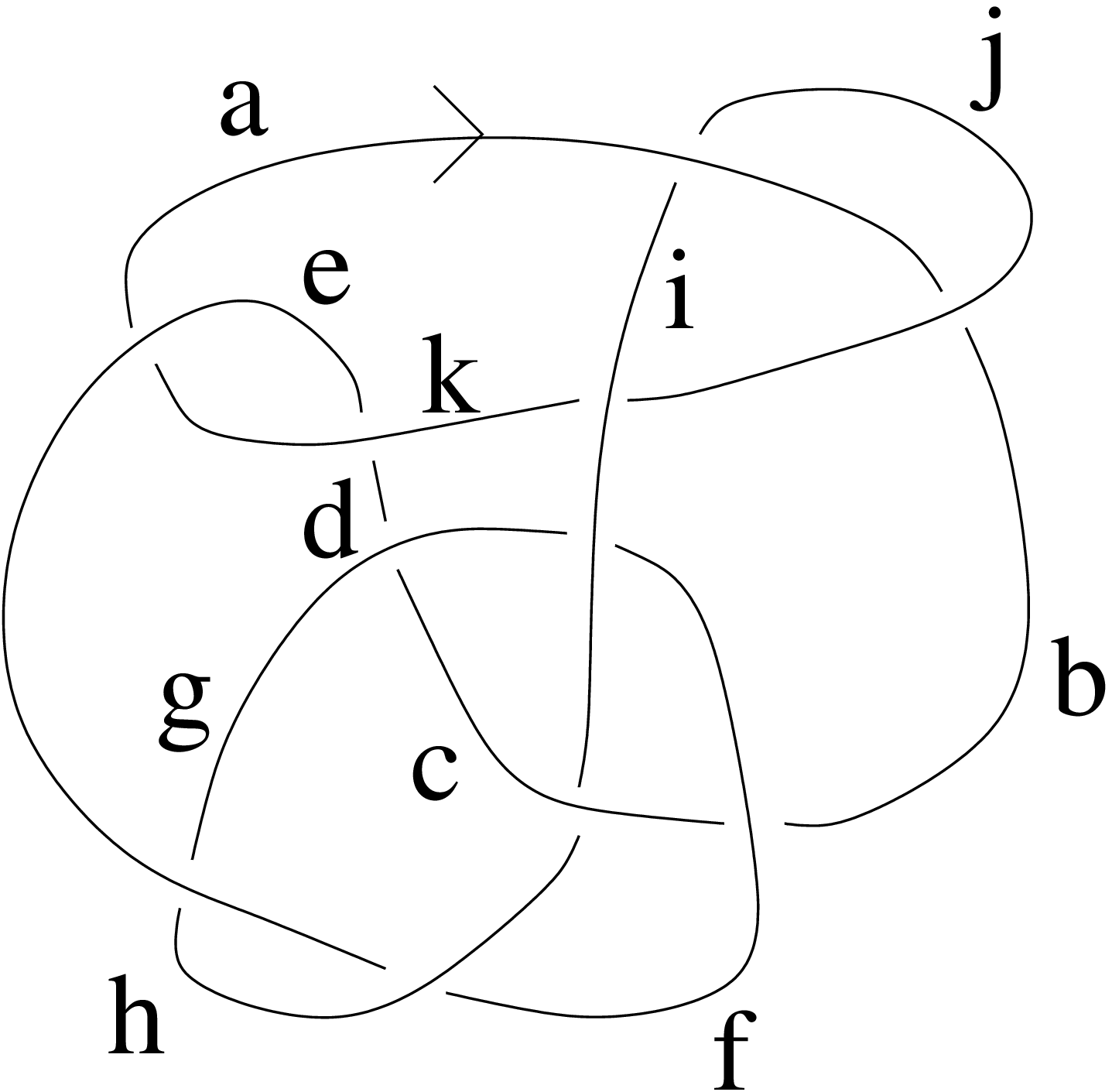}&\hspace{1cm}
\includegraphics[scale=0.28]{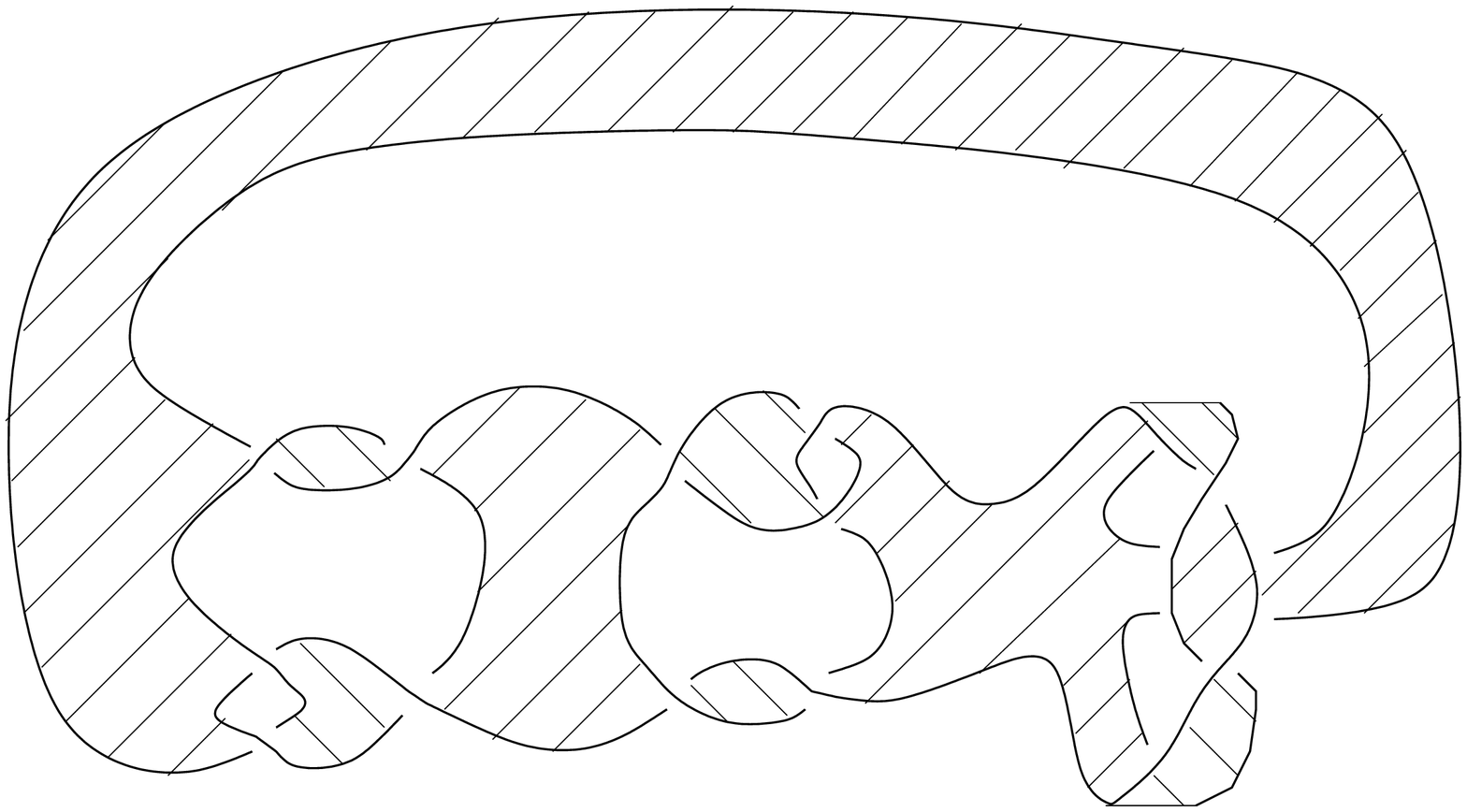}
\end{tabular}
\caption{The Conway knot
$11_{401}$ and a Seifert surface of genus 3 (from \cite{Ga84}).} \label{example11n34}
\end{center}
 \end{figure}
It is well--known that the genus of the Conway knot is 3 and that $\Delta_K=1$.
Since $\partial (S^3\sm \nu K)$ is an incompressible torus in $T_K$ it follows from the additivity of the Thurston norm that
 $\chi^-_{T_K}(\phi)=5+1=6$.

According to the following proposition the adjunction inequality does not determine $\chi^-_{M_K(e)}$, but for certain $\s$ the twisted Alexander polynomials of Theorem \ref{thm:joint} detect $\chi^-_{M_K(e)}(\s)$.

\begin{proposition} \label{prop:ex} Let $K \subset S^3$ be the Conway knot.
Then
\bn
\item $SW_{M_K(e)}=1$,
\item For any $\s$ with $p_*(\s)=\phi\in H^1(T_K;\Z)$ we have
\[ \chi^-_{M_K(e)}(\s)=\chi^-_{T_K}(p_*(\s))+|\s\cdot \s|=6+|\s\cdot \s|.\]
\en
\end{proposition}

\begin{proof}
Note that if $J$ is the unknot, then $T_J=T$. Now we use another Alexander polynomial one knot, namely $K$ to build $T_K$. It is not difficult to see (using e.g.
Mayer--Vietoris sequences) that $\Delta_{T_K}=\Delta_{T_J}=\Delta_{T}=1$. It follows from Theorem \ref{thm:mt} and Theorem \ref{thm:baldridge} that $SW_{M_K(e)}=1$.

Now let  $\s\in H^2(M_K(e);\Z)$ with $p_*(\s)=\phi \in H^1(T_K;\Z)$.
First observe
that by obstruction theory we can define a proper map $T \setminus \nu C \to S^1 \times D^2$
which realizes the map $\pi_1(T\sm \nu C)\to H_1(T\sm \nu C)\to \Z=H_1(S^1\times D^2)$ given by  $\phi\in H^1(T\sm \nu C;\Z)=\hom(\pi_1(T\sm \nu C),\Z)$  and which extends the diffeomorphism $\partial (T \setminus \nu
C) \cong S^1 \times
\partial D^2$ given by identifying $\lambda_C$ and $\mu_{C}$ with $S^1$ and $\partial D^2$
respectively. Out of this we construct a degree one map \[ T_{K} = (T\sm \nu C) \cup (S^3\sm \nu K) \to
(S^1 \times D^2) \cup (S^3\sm \nu K) = N_{K} \]  which induces an epimorphism $ \pi_1(T^3_K) \to \pi_1(N_K) $. Given an epimorphism
$\alpha: \pi_1(N_{K}) \to G$ onto a finite group we obtain a corresponding epimorphism of
$\pi_1(T_K)$ onto $G$, that we denote for simplicity by $\alpha$ as well.

We now get the following Mayer--Vietoris sequence (cf. \cite{FK06})
\[ \ba{cccccccccccccccc}
\hspace{-0.1cm}\hspace{-0.1cm}&\hspace{-0.1cm}\hspace{-0.1cm}\to\hspace{-0.1cm}&\hspace{-0.1cm}H_1(S^1\times S^1;\Z[G]\xpm)\hspace{-0.1cm}&\hspace{-0.1cm}\to\hspace{-0.1cm}&\hspace{-0.1cm}\ba{c} H_1(T^3\sm \nu C;\Z[G]\xpm)\\ H_1(S^3\sm \nu K;\Z[G]\xpm)\ea \hspace{-0.1cm}&\hspace{-0.1cm}\to\hspace{-0.1cm}&\hspace{-0.1cm}H_1(T_K;\Z[G]\xpm)\hspace{-0.1cm}&\hspace{-0.1cm}\to\hspace{-0.1cm}&\hspace{-0.1cm}\\[2mm]
\hspace{-0.1cm}&\hspace{-0.1cm}\to\hspace{-0.1cm}&\hspace{-0.1cm}H_0(S^1\times S^1;\Z[G]\xpm)\hspace{-0.1cm}&\hspace{-0.1cm}\to\hspace{-0.1cm}&\hspace{-0.1cm}\ba{c} H_0(T^3\sm \nu C;\Z[G]\xpm)\\ H_0(S^3\sm \nu K;\Z[G]\xpm) \ea \hspace{-0.1cm}&\hspace{-0.1cm}\to\hspace{-0.1cm}&\hspace{-0.1cm}H_0(T_K;\Z[G]\xpm)\hspace{-0.1cm}&\hspace{-0.1cm}\to\hspace{-0.1cm}&\hspace{-0.1cm} \ea \]
First note that the maps
\[ \ba{rcl} \phi\times \a: \pi_1(S^1\times S^1)&\to& \Z\times G\\
 \phi\times \a: \pi_1(T^3\sm \nu C)&\to& \Z\times G \ea \]
 both factor through $\phi$. In particular we get a commutative diagram
\[ \xymatrix{ H_i(S^1\times S^1;\Z[G]\xpm) \ar[d]^\cong \ar[r]^{\iota} &H_i(T^3\sm \nu C;\Z[G]\xpm)  \ar[d]^\cong \\
 H_i(S^1\times S^1;\Z\xpm) \otimes_\Z \Z[G]  \ar[r]^{\iota} & H_i(T^3\sm \nu C;\Z\xpm) \otimes_\Z \Z[G] .} \]
 It is now easy to see that the bottom map is an isomorphism for $i=0$.
 Furthermore $ H_1(S^1\times S^1;\Z\xpm) $ is given by $\Z c$ where $c$ is a curve representing the commutator $ [y,z]$.
On the other hand    $ H_1(T^3\sm \nu C;\Z\xpm) $ is given by $\Z y\oplus \Z z$. We see that the inclusion induced map
 $ H_1(S^1\times S^1;\Z\xpm) \to H_1(T^3\sm \nu C;\Z\xpm) $ is trivial. It follows from this discussion that the above Mayer--Vietoris sequence descends to a sequence of the form:
 \[ 0 \to \Z^{|G|} \oplus \Z^{|G|} \oplus H_1(S^3\sm \nu K;\Z[G]\xpm) \to H_1(T_K;\Z[G]\xpm) \to 0.\]
It follows that
\[ \ba{rcl} \deg(\Delta_{T_K,\phi}^\a)&=&\mbox{rank}_\Z(H_1(T_K;\Z[G]\xpm))\\
&=&\mbox{rank}_\Z(H_1(S^3\sm \nu K;\Z[G]\xpm))+2|G|\\
&=& \deg(\Delta_{K,\phi}^\a)+2|G|.\ea \]

We now pick a particular epimorphism $\a:\pi_1(N_K)\to G$.
The fundamental group
$\pi_1(N_K)$ is generated by the meridians $a,b,\dots,k$ of the
segments in the knot diagram of Figure \ref{example11n34}. The
relations are \[ \ba{rclrclrclrclrcl}
a&=&jbj^{-1},&    b&=&fcf^{-1},&    c&=&g^{-1}dg,&    d&=&k^{-1}ek,&   \\
e&=&h^{-1}fh,&    f&=&igi^{-1},&    g&=&e^{-1}he,&    h&=&c^{-1}ic,\\
i&=&aja^{-1},&    j&=&iki^{-1},&    k&=&e^{-1}ae,&  a^{-1}&=&jfg^{-1}k^{-1}h^{-1}ie^{-1}c^{-1}aie^{-1}. \ea \] Using the program \emph{KnotTwister}
\cite{Fr07}  we   found the homomorphism  $\varphi:\pi_1(N_K)\to A_5$ given by
 \[
\ba{rclrclrclrcl}
a&\mapsto&(142),& b&\mapsto&(451),& c&\mapsto&(451),& d&\mapsto&(453),\\
e&\mapsto&(453),& f&\mapsto&(351),& g&\mapsto&(351),& h&\mapsto&(431),\\
i&\mapsto&(351),& j&\mapsto&(352),& k&\mapsto&(321),&\ea \]
 where we use cycle notation.
%The generators of $\pi_1(X(K))$ are sent to the element in $S_5$
%given by the cycle with the corresponding capital letter.
Using
\emph{KnotTwister} we compute that the degree of $\Delta_{K,\phi}^\a$ computed as an element in $\F_{53}\tpm$ equals 209. It follows from \cite[Proposition~6.1]{FV08} that the degree of  $\Delta_{K,\phi}^\a$ computed as an element in $\zt$ equals at least 209.
We also have $\div(\phi_\a)=1$. It now follows from Theorem \ref{thm:joint} that
\[ \ba{rcl} \chi^-_{M_K(e)}(\s) &\geq& \frac{1}{|A_5|}\left(\deg(\Delta_{T_K,\phi}^{\alpha})-2 \div \, \phi_{\a}\right)+|\s\cdot \s|\\[0.1cm]
&=& \frac{1}{60}\left(\deg(\Delta_{K,\phi}^{\alpha})+2\cdot 60-2\right)+|\s\cdot \s|\\[1mm]
&\geq& \frac{1}{60}\left(209+2\cdot 60-2\right)+|\s\cdot \s|\\[1mm]
&= &\frac{327}{60}+|\s\cdot \s|.\ea \]
Lemma \ref{lem:starmin} asserts that $\chi^-_{M_K(e)}(\s) \leq 6 + |\s\cdot \s|$, as $\chi_{T_K}(\phi)$ equals $6$.
The claim now follows from Lemma \ref{lem:starmin} and the fact that $\chi^-_{M_K(e)}$ is necessarily an integer.
\end{proof}

%==========================================================
\subsection{The limitations of Theorem \ref{thm:joint}}
The following lemma says that in many cases all twisted Alexander polynomials will be zero and therefore Theorem \ref{thm:joint} will not be able to give any information on the minimal complexity of surfaces.

\begin{lemma} \label{lem:bad}
Let $N$ be a 3--manifold and $\phi \in H^1(N;\Z)$ such that $\Delta_{N,\phi}=0$. Then
$\Delta_{N,\phi}^\a =0$ for all epimorphisms $\a:\pi_1(N)\to G$ to a finite group $G$.
\end{lemma}

As an example, note that for any $N$ which is the direct sum of $N_1,N_2$ with $b_1(N_i)\geq 1$ we have
 $\Delta_{N,\phi}=0$ for any $\phi \in H^1(N;\Z)$. This can be seen using a straightforward
 Mayer--Vietoris argument. Another example is given by any $N$ which is the 0--framed surgery on a
 boundary link.

\begin{proof}
First note that for an epimorphism  $\b:\pi_1(N)\to H$ to a finite group $H$ we have $\Delta_{N,\phi}^\b=0$ if and only if
\[ \dim_{\Q}(H_1(N;\Q[H]\tpm))=\infty.\]
By Maschke's theorem (cf. e.g. \cite[p.~216]{Ro96}) we have a direct sum decomposition of $\Q[G]$--modules
\[ \Q[G]=\Q \oplus V,\]
where $\Q$ is the $\Q[G]$--module with the trivial $G$--action. It follows that
\[ (H_1(N;\Q[G]\tpm)\cong (H_1(N;\Q\tpm)\oplus (H_1(N;V\otimes_\Q \Q\tpm).\]
The lemma now follows immediately from the above observation applied to $H=e$ and $H=G$.
\end{proof}

%\begin{lemma} \ref{lem:bad}
%Assume we are given a 3--manifold $N$, $e\in H^2(N;\Z)$ and $\phi \in H^1(N;\Z)$ with $e\cup \phi=0$ and $\tnphi>0$.
%If $\deg(\Delta_{N,\phi})-2 < \tnphi$ then there exists  $C\in \N$ such that
%\[ \frac{1}{N}\left( \deg(\Delta_{N,C\phi}^\a)-2C\right) \leq \tnphi \]
%for all epimorphisms $\a:\pi_1(N)\to G$ to a finite group $G$.
%\end{lemma}

%==========================================================

%=========================================================

\end{document}